\documentclass[a4paper,twoside,11pt,reqno]{amsart}
\usepackage{fullpage}
\usepackage[T1]{fontenc}
\usepackage[utf8]{inputenc}

\usepackage{hyperref}
\usepackage[slantedGreeks, partialup, noDcommand]{kpfonts}
\usepackage{graphicx}
\usepackage[mathscr]{euscript}
\usepackage{caption}
\usepackage{booktabs}
\usepackage{tikz}
\usepackage{appendix}

\usetikzlibrary{trees}

\hypersetup{ocgcolorlinks=true,allcolors=testc}
\hypersetup{
	colorlinks   = true,
	citecolor    = black
}
\hypersetup{linkcolor=black}
\hypersetup{urlcolor=black}

\newtheorem{theorem}{Theorem}
\newtheorem{lemma}{Lemma}

\newtheorem*{definition*}{Definition}

\theoremstyle{definition}

\theoremstyle{remark}

\renewcommand{\epsilon}{\varepsilon}

\renewcommand{\phi}{\varphi}
\newcommand{\R}{\mathbb{R}}

\newcommand{\Z}{\mathbb{Z}}

\DeclareMathOperator{\E}{\mathbb{E}}

\def \R {\mathbb{R}}

\def \Z {\mathbb{Z}}

\def\bal{\begin{align*}}
\def\eal{\end{align*}}

\author{Xinyuan Xie}
\address{(X.~X.) Department of Mathematics, University of California, Irvine, CA 92697, USA}
\email{xinyuax7@uci.edu}

\title{Stability of Szarek's Inequality with Best Constant}

\thanks{}

\begin{document}

    \maketitle
	
    \begin{abstract}
Let $\varepsilon_1,\ldots,\varepsilon_n$ be independent Rademacher random variables. We settle the best constant $c$ such that
\[
 \E\left|\sum_{j=1}^n a_j\varepsilon_j\right|
 \ge\frac1{\sqrt2}+c\left|a-\frac{e_1+e_2}{\sqrt2}\right|
\]
holds for every $n\ge2$ and every unit vector $a\in\R^n$ with $a_1\ge\cdots\ge a_n\ge0$. We prove that the optimal constant is $c=(2-\sqrt2)^{3/2}/8$, with equality attained at $a=\frac12(1,1,1,1,0,\ldots,0)$ for $n\ge4$. This completes a line of research by De--Diakonikolas--Servedio, Eskenazis--Nayar--Tkocz and Fang--Wang.
\end{abstract}

\bigskip

{\footnotesize

}

\section{Introduction}
\subsection{Background}
Let $\varepsilon_1,\varepsilon_2,\ldots,\varepsilon_n$ be independent Rademacher random variables. Littlewood \cite{Lit30} conjectured in 1930 that, for every unit vector $a$ in $\R^n$,
\begin{equation}\label{eq:szarek}
 \E\left|\sum_{j=1}^n a_j\varepsilon_j\right|
 \ge\frac1{\sqrt2}.
\end{equation}
Szarek \cite{Sza76} proved this inequality in 1976, with equality attained when two coefficients have absolute value $1/\sqrt2$ and all the others vanish. Szarek's inequality is the sharp $L_1$ case of the $L_p$ Khinchin inequalities, whose study goes back to Khinchin's work on the law of the iterated logarithm \cite{Khi23}. Haagerup \cite{Haa81} gave a different proof of \eqref{eq:szarek} and also determined the best constants in the $L_p$ Khinchin inequalities for $p>1$. Subsequent proofs of Szarek's inequality were given by Toma\-szew\-ski \cite{Tom87} and Lata{\char170}a and Oleszkiewicz \cite{LO94}, whose argument also yields the sharp extension to vector-valued coefficients in arbitrary normed spaces. We refer to \cite[Section~3.3]{NT23} and \cite{ENTDist24} for further background.

A natural question is whether \eqref{eq:szarek} admits a quantitative improvement away from its equality cases. Motivated by questions on the noise stability of Boolean linear threshold functions, De, Diakonikolas and Servedio \cite[Theorem~9]{DDS16} proved that there exists a universal constant $c>0$ such that, for every $n\ge2$ and every unit vector $a\in\R^n$ with $a_1\ge\cdots\ge a_n\ge0$,
\begin{equation}\label{eq:robust-known}
 \E\left|\sum_{j=1}^n a_j\varepsilon_j\right|
 \ge\frac1{\sqrt2}+c\left|a-\frac{e_1+e_2}{\sqrt2}\right|.
\end{equation}
A related quantitative improvement was subsequently obtained by Melbourne and Roberto \cite{MR22}, but neither work gives an explicit value of $c$ in \eqref{eq:robust-known}.

Sharp Khinchin-type inequalities are closely related to extremal section and projection problems \cite{NT23}, as are their stability questions, which are studied by similar methods. Chasapis, Nayar and Tkocz \cite{CNT22} proved stability for extremal hyperplane sections of $\ell_p$ balls, using, in the cube case, a local self-improvement of Ball's inequality in the spirit of \cite{DDS16}. Jakimiuk \cite{Jak26b} subsequently established the dual counterpart for hyperplane projections.

Recently, Eskenazis, Nayar and Tkocz \cite{ENT24} obtained the first explicit constant $c=8\cdot10^{-5}$ in \eqref{eq:robust-known}, motivated by the resilience of the Ball--Szarek hyperplane for projections of $\ell_q$ balls. They also established distributional stability of Szarek's inequality under perturbations of the Rademacher distribution \cite{ENTDist24}. On the other hand, Jakimiuk \cite{Jak26a} recently established Gaussian stability for the sharp $L_p$ Khinchin inequalities with $p\ge3$, although the best stability constant at the endpoint $p=1$ remained unknown. Very recently, Fang and Wang \cite[Theorem~3]{FW26} improved the constant in \eqref{eq:robust-known} to $c=1/36$. In the next theorem, we settle the best constant.

\subsection{Main Results}

\begin{theorem}\label{thm:stable}
For every $n\ge2$ and every unit vector $a$ in $\R^n$ with $a_1\ge\cdots\ge a_n\ge0$, we have
\begin{equation}\label{eq:stable}
 \E\left|\sum_{j=1}^n a_j\varepsilon_j\right|
 \ge\frac1{\sqrt2}+c\left|a-\frac{e_1+e_2}{\sqrt2}\right|,
\end{equation}
where $c=(2-\sqrt2)^{3/2}/8$ is optimal. Equality is attained by $a=\frac12(1,1,1,1,0,\ldots,0)$ for $n\ge4$.
\end{theorem}

When $a_1+a_2\ge1$, Theorem~\ref{thm:stable} follows from Fang and Wang's estimate \cite[Theorem~1]{FW26}, whose proof is similar in spirit to the self-improvement argument in \cite[Section~5.1, Lemma~21]{ENT24}. The difficulty lies in the range $a_1+a_2\le1$, where Fang and Wang obtained $c=1/36$ by combining Haagerup's bound with the elementary estimate $\E|\sum_j a_j\varepsilon_j|\ge a_1$. We resolve this difficulty by proving a stronger lower bound in terms of Haagerup's function
\begin{equation}\label{eq:F}
 F(s)=\frac{2}{\sqrt{\pi s}}\,
       \frac{\Gamma((s+1)/2)}{\Gamma(s/2)},\qquad s>0,
\end{equation}
a bound which is tight for infinitely many values of $a_1+a_2$ tending to zero, see the next theorem.

\begin{theorem}\label{thm:stronger}
For every $n\ge2$ and every unit vector $a$ in $\R^n$ with $a_1\ge\cdots\ge a_n\ge0$ and $a_1+a_2\le1$, we have
\begin{equation}\label{eq:stronger}
 \E\left|\sum_{j=1}^n a_j\varepsilon_j\right|
 \ge F\!\left(\frac4{(a_1+a_2)^2}\right)
 \ge\frac34+\frac{1-(a_1+a_2)^2}{24}.
\end{equation}
The first inequality is attained when $2m$ nonzero coefficients are equal to $1/\sqrt{2m}$, for every integer $m\ge2$. \normalfont The proof proceeds as follows: applying the convexity of $s\mapsto F(1/s)$ to Haagerup's bound gives an estimate in terms of $\sum_j a_j^4$. If this estimate does not yield the first inequality in \eqref{eq:stronger}, then $a_2$ must be small relative to $a_1$. We then keep $\cos(a_1t)$ separate in the Fourier integral and compare the remaining factors with a Gaussian factor. The difficulty is to control the oscillatory tail beyond the interval on which these factors are nonnegative. Periodicity and symmetry reduce the cosine tail integrals to $[0,\pi/2]$, where a bound for the resulting kernel gives an error at most a constant times $a_2^2/\sqrt{1-a_1^2}$. The separation between $a_1$ and $a_2$ then makes this error small enough to prove \eqref{eq:stronger}.
\end{theorem}

\subsection{Organization of Proofs}
In Section~\ref{sec:proofs}, we first state Lemmas~\ref{lem:F} and \ref{lem:tail} without proof and use them to prove Theorem~\ref{thm:stronger} and Theorem~\ref{thm:stable}. We then prove the lemmas in Sections~\ref{subsec:F} and \ref{subsec:tail}, respectively.

\subsection{Evaluation of LLM evolution}
Frontier reasoning LLMs are now able to solve difficult open problems in mathematics with human assistance. We feel that in addition to the mathematical results they produce, the reasons why they discover these results so effectively are also of interest. Studying the evolution of the models and their reasoning trajectories may help us better understand the principles underlying their effectiveness. To address this question, we release two such instances documenting the evolution of three different versions of GPT Pro from January to May 2026, from being able to make small improvements to the constant to eventually being capable of determining its optimal value.

\begin{flushleft}
\url{https://chatgpt.com/share/6ac15a1a-c0e4-83e9-80de-be6766b749de}\par
\smallskip
\url{https://chatgpt.com/share/6ab5bbf3-4c38-83e9-a171-d50176882646}
\end{flushleft}

\begin{center}
\small
\renewcommand{\arraystretch}{1.15}
\begin{tabular}{@{}llll@{}}
\toprule
\textbf{Link} & \textbf{Turns} & \textbf{Model} & \textbf{Dates (2026, Pacific Time)} \\
\midrule
First link & 1--7 & GPT-5.2 Pro & January 24--26 \\
 & 8--17 & GPT-5.4 Pro & March 8--9 \\
 & 18--22 & GPT-5.5 Pro & April 27 \\
\addlinespace
Second link & 1--2 & GPT-5.4 Pro & March 17 \\
 & 3--11 & GPT-5.5 Pro & April 27--29 \\
 & 12 & GPT-5.4 Thinking & April 29 \\
 & 13--19 & GPT-5.5 Pro & April 29--May 1 \\
\bottomrule
\end{tabular}
\end{center}

The detailed reasoning summaries and additional execution details are accessible through the source code of the web page using Chrome Developer Tools or similar tools.

\section*{Acknowledgements}
The author is very grateful to Paata Ivanisvili for helpful discussions.

\section{Proofs}\label{sec:proofs}
Throughout this section, $a$ is a unit vector with $a_1\ge\cdots\ge a_n\ge0$. Zero coefficients are omitted whenever reciprocal powers occur.

\begin{lemma}\label{lem:F}
The function $F$ is increasing, while $s\mapsto F(1/s)$ is decreasing and convex on $(0,\infty)$ and extends continuously to zero with value $\sqrt{2/\pi}$. For every unit vector $a$ in $\R^n$,
\begin{equation}\label{eq:Haagerup}
 \E\left|\sum_{j=1}^n a_j\varepsilon_j\right|
 \ge\sum_{j=1}^n a_j^2F(a_j^{-2})
 \ge F\!\left(\frac1{\sum_{j=1}^n a_j^4}\right).
\end{equation}
Moreover,
\begin{equation}\label{eq:scalar}
 F(s)\ge\frac{19}{24}-\frac1{6s},\qquad s\ge4.
\end{equation}
\end{lemma}

\begin{lemma}\label{lem:tail}
For every unit vector $a$ in $\R^n$ with $a_1\ge\cdots\ge a_n\ge0$, $a_2>0$ and $a_1^2+3a_2^2\le1$, we have
\begin{equation}\label{eq:tail}
 \E\left|\sum_{j=1}^n a_j\varepsilon_j\right|
 \ge\sqrt{\frac2\pi}\left(1-\frac7{20}\frac{a_2^2}{\sqrt{1-a_1^2}}\right).
\end{equation}
\end{lemma}

\begin{proof}[Proof of Theorem~\ref{thm:stronger}]
If $a_2=0$, then $a_1=1$ and the expectation equals $1>F(4)=3/4$. Assume therefore that $a_2>0$.

\smallskip
\noindent\emph{Case 1.} Suppose that
\begin{equation}\label{eq:caseone}
 a_1^4+(1-a_1^2)a_2^2\le\frac{(a_1+a_2)^2}{4}.
\end{equation}
Since $\sum_{j=2}^n a_j^4\le a_2^2(1-a_1^2)$, Lemma~\ref{lem:F} gives
\[
 \E\left|\sum_{j=1}^n a_j\varepsilon_j\right|
 \ge F\!\left(\frac1{\sum_{j=1}^n a_j^4}\right)
 \ge F\!\left(\frac4{(a_1+a_2)^2}\right).
\]

\smallskip
\noindent\emph{Case 2.} Suppose that \eqref{eq:caseone} fails. We show that this forces $a_2$ to be small relative to $a_1+a_2$. Indeed,
\begin{align*}
 &4\bigl(a_1^4+(1-a_1^2)a_2^2\bigr)-(a_1+a_2)^2\\
 &\qquad=(a_1-a_2)\bigl(4a_1^2(a_1+a_2)-a_1-3a_2\bigr)>0.
\end{align*}
Thus $a_1>a_2$ and $4a_1^2(a_1+a_2)>a_1+3a_2$. Multiplying by $a_1+a_2$ and using $(a_1+a_2)^2\le1$, we obtain
\[
 3a_1^2-4a_1a_2-3a_2^2
 =4a_1^2-(a_1+a_2)(a_1+3a_2)>0.
\]
Thus $r=a_1/a_2$ satisfies $3r^2-4r-3>0$, so $r>9/5$. Since $a_1+a_2\le1$,
\[
 1-a_1^2\ge1-\frac{a_1^2}{(a_1+a_2)^2}
          =\frac{a_2(2a_1+a_2)}{(a_1+a_2)^2}.
\]
Consequently,
\[
 \frac{a_2^4}{(1-a_1^2)(a_1+a_2)^4}
 \le\frac1{(2r+1)(r+1)^2}<\frac1{36}.
\]
In particular, $a_2^2/\sqrt{1-a_1^2}<(a_1+a_2)^2/6$. Since
$a_1^2+3a_2^2\le(a_1+a_2)^2\le1$, we may now apply Lemma~\ref{lem:tail} and conclude that
\begin{align*}
 \E\left|\sum_{j=1}^n a_j\varepsilon_j\right|
 &>\sqrt{\frac2\pi}\left(1-\frac7{120}(a_1+a_2)^2\right)\\
 &>\bigl(1-(a_1+a_2)^2\bigr)\sqrt{\frac2\pi}
   +\frac34(a_1+a_2)^2
 \ge F\!\left(\frac4{(a_1+a_2)^2}\right).
\end{align*}
Here the second step uses $\sqrt{2/\pi}>90/113$, which follows from $\pi<22/7$, and the last follows from the convexity of $s\mapsto F(1/s)$ and $F(4)=3/4$.

This proves the first inequality in \eqref{eq:stronger}. The second follows from \eqref{eq:scalar}, evaluated at $s=4/(a_1+a_2)^2$.

To check sharpness of the first inequality, take $2m$ coefficients equal to $1/\sqrt{2m}$, $m\ge2$. In Haagerup's Fourier argument, all terms in the arithmetic--geometric mean inequality then coincide and the product of the cosines is nonnegative, so the expectation equals $F(2m)=F(4/(a_1+a_2)^2)$.
\end{proof}

\begin{proof}[Proof of Theorem~\ref{thm:stable}]
Following \cite{ENT24}, write
\begin{equation}\label{eq:delta}
 \delta(a)=\left|a-\frac{e_1+e_2}{\sqrt2}\right|^2
          =2-\sqrt2\,(a_1+a_2).
\end{equation}
For $a_1+a_2\ge1$, Fang and Wang \cite[Theorem~1]{FW26} proved
\[
 \E\left|\sum_{j=1}^n a_j\varepsilon_j\right|
 \ge\frac1{\sqrt2}
 +\frac{\sqrt{1-\delta(a)/4}-\sqrt{\delta(a)}}{2\sqrt2}
       \sqrt{\delta(a)}.
\]
The coefficient of $\sqrt{\delta(a)}$ decreases on $0\le\delta(a)\le2-\sqrt2$ and equals $c$ at the right endpoint. Hence \eqref{eq:stable} holds in this range.

It remains to consider $a_1+a_2\le1$, or equivalently $2-\sqrt2\le\delta(a)<2$. By Theorem~\ref{thm:stronger}, it suffices to show that
\[
 \frac{19-\tfrac12(2-t)^2}{24}-\frac1{\sqrt2}-c\sqrt t\ge0,
 \qquad 2-\sqrt2\le t\le2.
\]
The left-hand side is concave in $t$, since its second derivative satisfies
\[
 -\frac1{24}+\frac{c}{4t^{3/2}}
 \le-\frac1{24}+\frac1{32}=-\frac1{96}<0.
\]
It vanishes at $t=2-\sqrt2$ and is positive at $t=2$. For the latter assertion, note that $19/24-1/\sqrt2>0$ and
\[
 \left(\frac{19}{24}-\frac1{\sqrt2}\right)^2-2c^2
 =\frac{17(3-2\sqrt2)^2}{576}>0.
\]
Concavity proves the desired inequality on the whole interval. For $a=\frac12(1,1,1,1,0,\ldots,0)$, the expectation is $3/4$ and $\delta(a)=2-\sqrt2$, giving equality in \eqref{eq:stable} and proving optimality.
\end{proof}

\subsection{Haagerup's bound and the special function}\label{subsec:F}
We now prove Lemma~\ref{lem:F}. Haagerup's argument \cite{Haa81}, recalled in \cite[Section~4.2, proof of Theorem~32]{NT23}, begins with the Fourier-integral identity
\[
 |x|=\frac2\pi\int_0^\infty\frac{1-\cos(tx)}{t^2}\,dt.
\]
Taking expectations and using independence, we obtain
\begin{equation}\label{eq:fourier}
 \E\left|\sum_{j=1}^n a_j\varepsilon_j\right|
 =\frac2\pi\int_0^\infty\left(1-\prod_{j=1}^n\cos(a_jt)\right)\frac{dt}{t^2}.
\end{equation}

Since $\sum_j a_j^2=1$, the arithmetic--geometric mean inequality yields
\[
 \prod_j\cos(a_jt)\le\prod_j|\cos(a_jt)|
 \le\sum_j a_j^2|\cos(a_jt)|^{1/a_j^2}.
\]
To evaluate the resulting integrals, use
$\sum_{k\in\Z}(t+k\pi)^{-2}=\sin^{-2}t$, which follows from Euler's product for $\sin t$. Using the periodicity and evenness of the numerator and then integrating by parts gives
\begin{equation}\label{eq:cosineintegral}
 \int_0^\infty\frac{1-|\cos t|^s}{t^2}\,dt
 =\int_0^{\pi/2}\frac{1-\cos^s t}{\sin^2t}\,dt
 =s\int_0^{\pi/2}\cos^s t\,dt,\qquad s>0.
\end{equation}
The beta integral therefore gives
\begin{equation}\label{eq:Fintegral}
 F(s)=\frac2{\pi\sqrt s}\int_0^\infty\frac{1-|\cos t|^s}{t^2}\,dt
     =\frac{2\sqrt s}{\pi}\int_0^{\pi/2}\cos^s t\,dt.
\end{equation}
Substituting the arithmetic--geometric mean bound into \eqref{eq:fourier} and changing variables in each term proves the first inequality in \eqref{eq:Haagerup}.

Formula \eqref{eq:Fintegral} gives $F(4)=3/4$ and, after the substitution $t=u/\sqrt s$, yields $F(s)\to\sqrt{2/\pi}$ as $s\to\infty$ by dominated convergence, using
\begin{equation}\label{eq:cosgaussian}
 0\le\cos t\le e^{-t^2/2},\qquad 0\le t\le\frac\pi2.
\end{equation}
This elementary bound follows by integrating $(\log\cos t)'=-\tan t\le-t$.

Logarithmic differentiation of Euler's product for the gamma function gives
\begin{align}
 (\log F)'(s)
 &=-\frac1{2s}+\sum_{k=0}^\infty
       \left(\frac1{s+2k}-\frac1{s+2k+1}\right)\notag\\
 &=\frac12\int_0^\infty e^{-st}\tanh(t/2)\,dt>0.
 \label{eq:logderivative}
\end{align}
Thus $F$ is increasing (see \cite[Section~4.2, proof of Theorem~32]{NT23} for an alternative proof). Integrating from $1/s$ to infinity and changing variables, we also obtain
\begin{equation}\label{eq:reciprocal}
 \log\frac{F(1/s)}{\sqrt{2/\pi}}
 =-\frac12\int_0^\infty e^{-t}\frac{\tanh(st/2)}t\,dt.
\end{equation}
Since $\tanh$ is increasing and concave, the right-hand side is decreasing and convex in $s$. Hence $s\mapsto F(1/s)$ is decreasing and logarithmically convex, in particular convex. Jensen's inequality with weights $a_j^2$ now proves the second inequality in \eqref{eq:Haagerup}.

To prove \eqref{eq:scalar}, note that $\tanh t\le t$ and
$(\tanh t-t+t^3/3)'=t^2-\tanh^2t\ge0$ imply $\tanh t\ge t-t^3/3$. By \eqref{eq:logderivative},
\[
 (\log F)'(s)\ge\frac1{4s^2}-\frac1{8s^4}.
\]
For $s\ge4$, use $F(s)\ge F(4)=3/4$ to get
\[
 F'(s)\ge\frac34\left(\frac1{4s^2}-\frac1{8s^4}\right)
 \ge\frac1{6s^2}.
\]
Integrating from $4$ to $s$ proves \eqref{eq:scalar} and completes the proof of Lemma~\ref{lem:F}. \hfill$\square$

\subsection{The estimate for the smaller coefficients}\label{subsec:tail}
For $0\le t\le\pi/(2a_2)$, the factors $\cos(a_jt)$ with $j\ge2$ are nonnegative and satisfy \eqref{eq:cosgaussian}. The next two lemmas give a lower bound for the Gaussian comparison integral and an upper bound for the cosine tails. Here $x_+=\max\{x,0\}$.

\begin{lemma}\label{lem:positive}
For every $0\le x\le1$,
\begin{equation}\label{eq:positive}
 \frac2\pi\int_0^\infty
 \frac{1-(\cos(xt))_+e^{-(1-x^2)t^2/2}}{t^2}\,dt
 \ge\sqrt{\frac2\pi}.
\end{equation}
\end{lemma}

\begin{proof}
The case $x=0$ is an equality. For $x>0$, we subtract the Gaussian integral
\[
 \frac2\pi\int_0^\infty\frac{1-e^{-t^2/2}}{t^2}\,dt=\sqrt{\frac2\pi}
\]
from the left-hand side of \eqref{eq:positive} and make the change of variables $t=u/x$. It then suffices to prove
\begin{equation}\label{eq:weighted-comparison}
 \int_0^\infty e^{-\lambda t^2}
 \frac{(\cos t)_+-e^{-t^2/2}}{t^2}\,dt\le0,
 \qquad \lambda\ge0,
\end{equation}
and then take $\lambda=(1-x^2)/(2x^2)$. 

By \eqref{eq:cosgaussian}, the numerator in \eqref{eq:weighted-comparison} is nonpositive on $[0,3\pi/2]$. Moreover, $(\cos t)_+=(\cos t+|\cos t|)/2$, so the Fourier identity and \eqref{eq:cosineintegral} at $s=1$ give
\[
 \int_0^\infty\frac{(\cos t)_+-e^{-t^2/2}}{t^2}\,dt
 =\sqrt{\frac\pi2}-\frac\pi4-\frac12
 =-\frac{(\sqrt\pi-\sqrt2)^2}{4}<-\frac1{49}.
\]
In the last step, $3<\pi<4$ and $\sqrt2<3/2$ give
$(\sqrt\pi-\sqrt2)^2/4=(\pi-2)^2/[4(\sqrt\pi+\sqrt2)^2]>1/49$.
The Gaussian part of the tail satisfies
\[
 \int_{3\pi/2}^\infty\frac{e^{-t^2/2}}{t^2}\,dt
 \le\frac{e^{-9\pi^2/8}}{(3\pi/2)^3}<\frac8{729}<\frac1{49}.
\]

Since the numerator in \eqref{eq:weighted-comparison} is nonpositive on $[0,3\pi/2]$ and $e^{-\lambda t^2}$ is nonincreasing, replacing the latter by $e^{-9\lambda\pi^2/4}$ gives an upper bound on this interval. On the tail, the same replacement gives an upper bound after the negative Gaussian term has been discarded. Thus
\begin{align*}
 &\int_0^\infty e^{-\lambda t^2}
       \frac{(\cos t)_+-e^{-t^2/2}}{t^2}\,dt\\
 &\qquad\le e^{-9\lambda\pi^2/4}
 \left[-\frac{(\sqrt\pi-\sqrt2)^2}{4}
       +\int_{3\pi/2}^\infty\frac{e^{-t^2/2}}{t^2}\,dt\right]<0.
\end{align*}
This proves \eqref{eq:weighted-comparison} and hence the lemma.
\end{proof}

\begin{lemma}\label{lem:cosinetail}
For every $s>1/2$,
\begin{equation}\label{eq:cosinetail}
 \int_{\pi/2}^\infty\frac{|\cos t|^s}{t^2}\,dt
 \le\frac13\sqrt{\frac\pi{2s}}.
\end{equation}
\end{lemma}

\begin{proof}
Using periodicity and evenness as in \eqref{eq:cosineintegral}, we obtain
\[
 \int_{\pi/2}^\infty\frac{|\cos t|^s}{t^2}\,dt
 =\int_0^{\pi/2}\cos^s t\left(\frac1{\sin^2t}-\frac1{t^2}\right)\,dt.
\]
We first prove the estimate
\begin{equation}\label{eq:kernel}
 \sqrt{\cos t}\left(\frac1{\sin^2t}-\frac1{t^2}\right)\le\frac13,
 \qquad 0<t<\frac\pi2.
\end{equation}
We use the Cusa--Huygens inequality $\sin t/t\le(2+\cos t)/3$, which follows since $t(2+\cos t)-3\sin t$ vanishes at zero and has derivative
\[
 2\sin(t/2)\bigl(2\sin(t/2)-t\cos(t/2)\bigr)\ge0.
\]
Differentiation then yields
\begin{align*}
 &2\sqrt{\cos t}\,\sin^3t\,\frac{d}{dt}
    \left[\sqrt{\cos t}\left(\frac1{\sin^2t}-\frac1{t^2}\right)\right]\\
 &\qquad=4\cos t\left(\frac{\sin t}{t}\right)^3
       +\sin^2t\left(\frac{\sin t}{t}\right)^2-1-3\cos^2t\\
 &\qquad\le-\frac{(1-\cos t)^3(\cos t+15)}{27}\le0.
\end{align*}
The first inequality follows by substituting $(2+\cos t)/3$ for $\sin t/t$ in the increasing polynomial on the preceding line. Thus the left-hand side of \eqref{eq:kernel} is nonincreasing, and its limit at zero is $1/3$. This proves \eqref{eq:kernel}.

Kershaw's inequality \cite[(1.3)]{Ker83}, applied with $x=s/2$ and parameter $1/2$, gives
\begin{equation}\label{eq:gamma-upper}
 F(s)\le\sqrt{\frac2\pi}\sqrt{\frac{s}{s+1/2}},\qquad s>0.
\end{equation}
Combining \eqref{eq:kernel}, \eqref{eq:Fintegral} and \eqref{eq:gamma-upper}, we obtain
\begin{align*}
 \int_{\pi/2}^\infty\frac{|\cos t|^s}{t^2}\,dt
 &\le\frac13\int_0^{\pi/2}\cos^{s-1/2}t\,dt
 =\frac{\pi F(s-1/2)}{6\sqrt{s-1/2}}\\
 &\le\frac13\sqrt{\frac\pi{2s}}.\qedhere
\end{align*}
\end{proof}

\begin{proof}[Proof of Lemma~\ref{lem:tail}]
We split the integral in \eqref{eq:fourier} at $\pi/(2a_2)$. On the first interval, \eqref{eq:cosgaussian} gives
\[
 \prod_{j=1}^n\cos(a_jt)
 \le(\cos(a_1t))_+e^{-(1-a_1^2)t^2/2}.
\]
For the tail, the arithmetic--geometric mean inequality with weights $a_j^2/(1-a_1^2)$, $j\ge2$, gives
\[
 \prod_{j=2}^n|\cos(a_jt)|
 \le\sum_{j=2}^n\frac{a_j^2}{1-a_1^2}
             |\cos(a_jt)|^{(1-a_1^2)/a_j^2}.
\]
When we subtract the comparison integral in Lemma~\ref{lem:positive} from \eqref{eq:fourier}, the first interval contributes nonnegatively. On the tail, we discard the nonnegative comparison term and use $|\cos(a_1t)|\le1$ to obtain
\begin{equation}\label{eq:tailstart}
 \E\left|\sum_{j=1}^n a_j\varepsilon_j\right|
 \ge\sqrt{\frac2\pi}
 -\frac2\pi\sum_{j=2}^n\frac{a_j^2}{1-a_1^2}
 \int_{\pi/(2a_2)}^\infty
 \frac{|\cos(a_jt)|^{(1-a_1^2)/a_j^2}}{t^2}\,dt.
\end{equation}

We split each integral on the right at $\pi/(2a_j)\ge\pi/(2a_2)$, applying \eqref{eq:cosgaussian} before this point and, after a change of variables, Lemma~\ref{lem:cosinetail} afterwards. Since $(1-a_1^2)/a_j^2\ge3$, we obtain
\begin{align*}
 &\int_{\pi/(2a_2)}^\infty
 \frac{|\cos(a_jt)|^{(1-a_1^2)/a_j^2}}{t^2}\,dt\\
 &\qquad\le\int_{\pi/(2a_2)}^\infty\frac{e^{-(1-a_1^2)t^2/2}}{t^2}\,dt
     +\frac13\sqrt{\frac\pi2}\frac{a_j^2}{\sqrt{1-a_1^2}}.
\end{align*}
Substituting in \eqref{eq:tailstart} and using $\sum_{j=2}^n a_j^4\le a_2^2(1-a_1^2)$, we arrive at
\begin{align}\label{eq:tailalmost}
 \E\left|\sum_{j=1}^n a_j\varepsilon_j\right|
 &\ge\sqrt{\frac2\pi}
       \left(1-\frac{a_2^2}{3\sqrt{1-a_1^2}}\right)\notag\\
 &\qquad-\frac2\pi\int_{\pi/(2a_2)}^\infty
                    \frac{e^{-(1-a_1^2)t^2/2}}{t^2}\,dt.
\end{align}

It remains to absorb the Gaussian remainder. On its range of integration,
$t^{-2}\le t/\bigl(\pi/(2a_2)\bigr)^3$, and therefore
\[
 \frac2\pi\int_{\pi/(2a_2)}^\infty
        \frac{e^{-(1-a_1^2)t^2/2}}{t^2}\,dt
 \le\frac{16a_2^3}{\pi^4(1-a_1^2)}
           e^{-\pi^2(1-a_1^2)/(8a_2^2)}.
\]
Dividing the last expression by $\sqrt{2/\pi}\,a_2^2/\sqrt{1-a_1^2}$ gives
\[
 \frac8{\pi^3}\sqrt{\frac2\pi}\frac{a_2}{\sqrt{1-a_1^2}}
 e^{-\pi^2(1-a_1^2)/(8a_2^2)}
 <\frac{16}{81}\left(\frac38\right)^3
 =\frac1{96}<\frac1{60}.
\]
Here we used $(1-a_1^2)/a_2^2\ge3$, $\pi>3$ and $e>8/3$: they imply $a_2/\sqrt{1-a_1^2}<2/3$ and bound the exponential by $e^{-3}<(3/8)^3$. Combining this with \eqref{eq:tailalmost} and $1/3+1/60=7/20$ proves \eqref{eq:tail}.
\end{proof}

\end{document}